\documentclass[11pt,reqno]{amsart}

\usepackage[utf8]{inputenc} 

\usepackage[margin=1in]{geometry} 

\usepackage{graphicx} 
\usepackage{float} 

\usepackage{changepage}

 \usepackage[parfill]{parskip} 
 
\usepackage{booktabs} 
\usepackage{array} 
\usepackage{paralist} 
\usepackage{verbatim} 
\usepackage{subfig} 
\usepackage{mathrsfs}
\usepackage{amssymb}
\usepackage{xcolor}
\usepackage{amsthm}
\usepackage{amsmath,amsfonts,amssymb,esint,hyperref}
\usepackage[noabbrev, capitalize]{cleveref}

\usepackage{autonum}

\usepackage{graphics,color}
\usepackage{enumerate, enumitem}
\usepackage{mathtools,centernot}
\usepackage{cases}
\usepackage{amsrefs}
\usepackage{bbm}
\usepackage{xfrac}

\usepackage{bookmark}

\newtheorem{theorem}{Theorem}[section]
\newtheorem{lemma}[theorem]{Lemma}

\newtheorem{corollary}[theorem]{Corollary}
\newtheorem{definition}[theorem]{Definition}
\newtheorem{proposition}[theorem]{Proposition}
\newtheorem{remark}[theorem]{Remark}

\numberwithin{equation}{section} 

\newcommand{\norm}[1]{\left\|#1\right\|}
\newcommand{\abs}[1]{\left|#1\right|}

\newcommand*{\R}{\ensuremath{\mathbb{R}}}

\newcommand*{\N}{\ensuremath{\mathbb{N}}}

\newcommand{\eps}{\varepsilon}
\newcommand*{\tr}{\ensuremath{\mathrm{tr\,}}}

\newcommand{\e}{\varepsilon}

\newcommand{\quotes}[1]{``#1''}
\renewcommand{\phi}{\varphi}

\renewcommand{\MR}[1]{} 

\usepackage{color, graphicx}
\usepackage{mathrsfs, dsfont}

\usepackage[]{hyperref}
\hypersetup{
    colorlinks=true,       
    linkcolor=red,          
    citecolor=blue,        
    filecolor=red,      
    urlcolor=cyan           
}

\def\dist{\mathop{\rm dist}\nolimits}    
\def\div{\mathop{\rm div}\nolimits}    
 
\def\spt{\mathop{\rm Spt}\nolimits} 
\def\tr{\mathop{\rm Tr}\nolimits}

\newcommand{\be}{\begin{equation}}
\newcommand{\ee}{\end{equation}}

\title{Dissipation for codimension 1 singular structures in the incompressible Euler equations}

\author[L. De Rosa]{Luigi De Rosa}
\address[L. De Rosa]{Gran Sasso Science Institute, viale Francesco Crispi, 7, 67100 L’Aquila, Italy}
\email{luigi.derosa@gssi.it}

\author[M. Inversi]{Marco Inversi}
\address[M. Inversi]{Departement Mathematik und Informatik, Universit\"at Basel, CH-4051 Basel, Switzerland}
\email{marco.inversi@unibas.ch}

\author[M. Nesi]{Matteo Nesi}
\address[M. Nesi]{Departement Mathematik und Informatik, Universit\"at Basel, CH-4051 Basel, Switzerland}
\email{matteo.nesi@unibas.ch}

\date{\today}

\subjclass[2020]{35Q31 - 35D30 - 26A45 - 28A75.}
\keywords{Incompressible Euler - Dissipation measure - Onsager's conjecture - Normal traces.}
\thanks{\textit{Acknowledgments}.
 The authors acknowledge the support of the SNF grant FLUTURA: Fluids, Turbulence, Advection No. 212573}

\begin{document}

\begin{abstract}
We consider weak solutions to the incompressible Euler equations. It is shown that energy conservation holds in any Onsager critical class in which smooth functions are dense. The argument is independent of the specific critical regularity and the underlying PDE. This groups several energy conservation results  and it suggests that critical spaces where smooth functions are dense are not at all different from subcritical ones, although possessing the \quotes{minimal} regularity index. Then, we study properties of the dissipation $D$ in the case of bounded solutions that are allowed to jump on $\mathcal H^d$-rectifiable 
space-time sets $\Sigma$, which are the natural dissipative regions in the compressible setting. As soon as both the velocity and the pressure posses traces on $\Sigma$, it is shown that  $\Sigma$ is $D$-negligible. The argument makes the role of the incompressibility very apparent, and it prevents dissipation on codimension 1 sets even if they happen to be densely distributed. As a corollary, we deduce energy conservation for bounded solutions of \quotes{special bounded deformation}, providing the first energy conservation criterion in a critical class where only an assumption on the \quotes{longitudinal} increment is made, while the energy flux does not vanish for kinematic reasons.
\end{abstract}

\maketitle

\section{Introduction}

Let $\Omega\subset \R^d$ be any open set. We consider the incompressible Euler equations
\begin{equation}\label{E}
\left\{\begin{array}{l}
\partial_t u +\div (u \otimes u) +\nabla p=f \\
\div u=0
\end{array}\right. \qquad \text{in } \Omega \times (0,T).
\end{equation}
 The initial datum and the boundary condition are not specified since our analysis will be focused on the interior of $\Omega \times (0,T)$. In particular, all the regularity assumptions used in this note should be intended to be local, i.e. in any compact subset of $\Omega\times (0,T)$.

Following the Onsager's deterministic ideal  picture \cite{O49}, the study of turbulent flows and their link with the emergence of anomalous dissipation phenomena is subject to the analysis of weak solutions to \eqref{E} which fail to conserve the kinetic energy. After Eyink and Constantin--E--Titi \cites{E94,CET94}, a clear mathematical framework has been settled in \cite{DR00}. For any $u\in L^3_{x,t}$ and $f\in L^{\sfrac{3}{2}}_{x,t}$ it holds\footnote{More generally, given $u$ in a certain class, any assumption on the external force which makes  $f\cdot u$ well-defined suffices.}
\begin{equation}
    \label{local_energy_eq}
\partial_t \frac{\abs{u}^2}{2} +\div  \left(\left(\frac{\abs{u}^2}{2}+p\right)u\right)=-D + f\cdot u \qquad \text{in } \mathcal D'_{x,t},
\end{equation}
where $D=\lim_{\ell\rightarrow 0} D^\ell_{\rm DR}$ in the sense of distributions, with\footnote{Here we are tacitly assuming that $x$ is picked in any compact subset of $\Omega$ and $\ell$ is sufficiently small so that $x+\ell z\in \Omega$.}
\begin{equation}\label{approx_of_D}
D^\ell_{\rm{DR}}:=  \int \nabla \rho (z) \cdot \frac{\delta_{\ell z} u(x,t)}{4\ell}   \abs{ \delta_{\ell z} u(x,t) }^2 dz, \qquad \delta_{\ell z} u (x,t):=u(x+\ell z,t)-u(x,t).
\end{equation}
Here $\rho$ is any smooth, even, compactly supported kernel but nonetheless the limit $D$ is independent of it. Different approximations of $D$ are available, giving slightly different insights on the dissipation. Indeed, following the original Constantin--E--Titi regularization, we also have $D=\lim_{\ell\rightarrow 0} D^\ell_{\rm CET}$, with
\begin{equation}\label{approx_of_D_cet}
D^\ell_{\rm{CET}}:=  R_\ell:\nabla u_\ell, \qquad R_\ell:=u_\ell\otimes u_\ell - (u\otimes u)_\ell,
\end{equation}
where $u_\ell$ is the space regularization of $u$ at scale $\ell$. Moreover, whenever $u$ arises as a strong $L^3_{x,t}$ limit of a vanishing viscosity sequence of suitable weak solutions to Navier--Stokes, then $D$ is a non-negative Radon measure (see \cite{DR00}). We point out that \eqref{approx_of_D} (or modifications of it) provides a deterministic counterpart of longitudinal, absolute or mixed third order \emph{structure functions} \`a la Kolmogorov. We refer to the recent work \cite{Novack24} where this has been proved under minimal assumptions. 

\begin{remark}
    \label{R:approx_are_trilin_operator}
Denote by $u(t):=u(\cdot,t)$. It is clear that both  sequences $D^\ell_{\rm DR}$ and $D^\ell_{\rm CET}$ from \eqref{approx_of_D} and \eqref{approx_of_D_cet} can be written as trilinear operators applied to the vector field $u$, that is  
\begin{equation}\label{T_DR}
D^\ell_{\rm DR}=T_{\rm DR}^\ell [u(t),u(t),u(t)] \quad \text{with} \quad T_{\rm DR}^\ell [v_1,v_2,v_3]:= \int \nabla \rho (z) \cdot \frac{\delta_{\ell z} v_1(x)}{4\ell}    \delta_{\ell z} v_2(x) \cdot \delta_{\ell z} v_3(x)  \,dz, 
\end{equation}
$$ D^\ell_{\rm CET}=T_{\rm CET}^\ell [u(t),u(t),u(t)] \quad \text{with} \quad T_{\rm CET}^\ell [v_1,v_2,v_3]:= ((v_1)_\ell\otimes (v_2)_\ell - (v_1\otimes v_2)_\ell):\nabla (v_3)_\ell. $$
\end{remark}

A direct consequence of \eqref{approx_of_D} is that $\norm{D^\ell_{\rm DR}}_{L^1_{x,t}}\rightarrow 0$ as soon as $u \in L^3_t B^{\theta}_{3,\infty}$ with $\theta>\sfrac{1}{3}$, proving local energy balance for solutions in any \quotes{Onsager subcritical} class. On the other hand, for $\theta<\sfrac{1}{3}$, the convex integration methods \cites{DS13,Is18,BDSV19,GR24} give H\"older continuous solutions with non-trivial dissipation $D$. Previous constructions at lower regularity were obtained in \cites{Shn00,Sch93}. Subsequent improvements  \cites{Is22,DK22} also provided H\"older continuous solutions with $D>0$, which eventually have been extended to the whole Onsager supercritical range in the $L^3$-based framework \cites{GKN24_1,GKN24_2} by Giri--Kwon--Novack. Although these results mathematically validate the Onsager prediction, the critical case $\theta=\sfrac{1}{3}$ remains open. In this direction, energy conservation is known if the solution, in addition to posses exact $\sfrac{1}{3}$ differentiability, has suitable decay of Littlewood–Paley pieces \cite{CCFS08} (i.e. $u\in L^3_tB^{\sfrac{1}{3}}_{3,c_0}$) or if an averaged Besov seminorm vanishes at small scales \cites{FW18,BGSTW19} (i.e. $u\in L^3_tB^{\sfrac{1}{3}}_{3, VMO}$). Moreover, by looking at the Constantin--E--Titi approximation \eqref{approx_of_D_cet}, it is apparent that energy conservation follows if $u\in L^\infty_t C^0_x$ with\footnote{Here we are using that $R_\ell : \nabla u_\ell= R_\ell: E u_\ell$ since $R_\ell$ is a symmetric matrix.} $E u :=\frac{\nabla u + \nabla u^T}{2} \in L^1_{x,t}$ or if $u\in L^\infty_t L^2_x$ with $Eu\in L^1_t L^\infty_x$. The reader may recognize that the latter condition guarantees weak strong uniqueness \cites{BDS11,DIS23,W18}. We want to investigate the common feature of those classes of solutions. A quick inspection shows that any of the above assumptions guarantees a uniform bound in $L^1_{x,t}$ on the approximating sequence of $D$ and, moreover, smooth functions are dense in at least one of the spaces involved (with respect to the spatial variable). This simple general principle is also sufficient. 

\begin{theorem}\label{T:main_smooth_dense}
Let $u\in L^3_{x,t}$ be any vector field and let $D^\ell_{\bullet}$ be any of the two sequences \eqref{approx_of_D}, \eqref{approx_of_D_cet}. Let $T^\ell_{\bullet}$ be the corresponding trilinear operator defined in \cref{R:approx_are_trilin_operator}. Let $C^\infty_x\subset X_1,X_2,X_3\subset L^1_x$ be three Banach spaces such that 
\begin{equation}
    \label{uniform_bound_T}
\sup_{\ell>0}\norm{T_{\bullet}^\ell[v_1,v_2,v_3]}_{L^1_x}\leq C \prod_{i=1}^3\norm{v_i}_{X_i}\qquad \forall v_i\in X_i, \, \text{ for } \, i=1,2,3.
\end{equation}
Moreover, assume that 
\begin{equation}\label{T_vanish_smooth}
\lim_{\ell\rightarrow 0} \norm{T_{\bullet}^\ell[v_1,v_2,v_3]}_{L^1_x}=0 \qquad \text{if } \quad v_i\in X_i,\, i=1,2,3 \quad \text{ and } \quad \exists i_0 \, \text{ s.t. } \, v_{i_0}\in C^\infty_x.
 \end{equation}
Then $D^\ell_{\bullet} \rightarrow 0$ in $L^1_{x,t}$ as soon as $u\in L^{p_1}_t X_1\cap L^{p_2}_tX_2\cap L_t^{p_3}X_3$ with  $\sfrac{1}{p_1}+ \sfrac{1}{p_2}+\sfrac{1}{p_3}=1$ and $C^\infty_x$ is dense in at least one of the  $X_i$, $i=1,2,3$.
\end{theorem}

The assumption \eqref{uniform_bound_T} means that we are considering spaces $X_i$ which are \quotes{at least} critical, while \eqref{T_vanish_smooth} is the natural requirement that the energy flux strongly vanishes if, in addition, any of the entries is smooth.
The above theorem follows from the simple observation that the trilinearity of the energy flux, together with the uniform bound \eqref{uniform_bound_T}, implies strong convergence to zero if $u$ can be approximated with more regular functions in any of the spatial norms involved. 
Indeed, these approximations (or say, \quotes{commutators}) of the defect distribution $D$ arise to point out the critical threshold but measuring it only at small scales. Thus, it is clear that imposing density of $C^\infty_x$ completely kills the mechanism that could allow for a non-trivial energy defect $D$. Let us emphasize that in the theorem above we are not claiming much originality with respect to the usual arguments, and the only goal is to single out the common mechanism of several results which otherwise could appear as different (see \cref{C: en cons euler general}). For instance, it does not even matter what is the additional property that functions (or their norm) in  $\overline{C^\infty}^{X_i}$ have with respect to the ones in $X_i$. The only message that \cref{T:main_smooth_dense} should give is that, given $u$ in a class which makes the energy flux bounded but non-vanishing, then $D\equiv 0$ in the subspace where smooth functions are dense without any additional effort and, perhaps more importantly, forgetting about the PDE from which $D$ arose. To the best of our knowledge, the statements $(ii), (vi)$ of \cref{C: en cons euler general} are new.

\cref{T:main_smooth_dense} considers only the case of trilinear structures, while in \cref{S:smooth_dense} the slightly more general setting of \quotes{multilinear operators} will be considered. These will apply, for instance, to the compressible \cites{FGSW17} and inhomogeneous \cites{LS16,NNT20} Euler setting, but also to the transport equations \cites{Ambr04,DipLi89} with rough vector fields. In the latter case our general approach recovers the renormalization property for $BD$ vector fields whenever the advected scalar is continuous (see \cref{S:transport_renorm}). Sticking to the context of the transport equations, density of smooth functions is indeed the conceptual difference between the DiPerna--Lions approach \cite{DipLi89} for Sobolev vector fields and the Ambrosio's one \cite{Ambr04} in $BV$. In the latter case, for merely bounded densities, $C^\infty_x$ fails to be dense in both $BV_x$ and $L^\infty_x$. The elegant idea of Ambrosio is to note that the $BV_x$ assumption on the vector field allows to optimize in the regularization kernel appearing in the defect distribution, from which the renormalization (or in our context, energy conservation) follows. This idea has been used in \cite{DRINV23} to prove energy conservation in the context of the incompressible Euler system for solutions $u\in L^\infty_{x,t}\cap L^1_tBV_x$. In that class, it is proved that $D\equiv 0$ although the strong $L^1_{x,t}$ convergence  $D_\bullet^\ell\rightarrow 0$ may fail , at least for a fixed choice of $\rho$. It must be noted that a careful space-time inhomogeneous choice of the kernel $\rho$ along the sequence $\ell\rightarrow 0$ allows a suitable refinement of $D^\ell_{\rm DR}$ to strongly vanish in $L^1_{x,t}$ (see \cite{Bonicatto}*{Theorem 2.4}). In any case, the incompressiblity\footnote{More generally, $\div u$ must be a measure absolutely continuous with respect to Lebesgue \cite{Ambr04}.} constraint plays a fundamental role, coherently with the non-trivial energy dissipation concentrating on codimension $1$ sets in the context of compressible shocks. 

The above discussion drives us towards the study of  codimension $1$ singular structures and their link to the dissipation measure $D$ in the incompressible setting. In this direction,  the first result has been obtained by Shvydkoy \cite{Shv09} where energy conservation for Vortex-Sheets (but not only) type solutions has been obtained. Vortex-Sheets are weak solutions to \eqref{E} with the velocity that experiences a jump in the tangential component across a (regular enough) hypersurface in space. The sheet is required to evolve regularly in time. Remarkably, in \cite{Shv09} the energy conservation is shown by looking at the left-hand side in \eqref{local_energy_eq}, while kinematic arguments fail. To the best of our knowledge, this has been the first instance where energy conservation is proved by relying on the structure of the PDE and not on the form of the energy flux. Towards a better understanding of these mechanisms, we propose the following result.

\begin{theorem}\label{T:main_traces}
Let $u,p\in L^\infty_{x,t}$ be a weak solution to \eqref{E} with forcing $f \in L^1_{x,t}$. Assume that both $u$ and $p$ have bilateral traces on Lipschitz space-time hypersurfaces in the sense of \cref{d: inner/outer trace}. Assume also that the Duchon--Robert distribution $D$ in \eqref{local_energy_eq} is a Radon measure. Then, $\abs{D}(\Sigma)=0$ for every countably $\mathcal{H}^d$-rectifiable set $\Sigma\subset \Omega\times (0,T)$.
\end{theorem}

We recall that $\Sigma$ is countably $\mathcal{H}^d$-rectifiable if, up to a $\mathcal H^d$-negligible set, it can be covered by a countable union of Lipschitz graphs (see \cref{S:gmt} below). In particular, although the solution is not assumed to be continuous, the existence of suitable traces prevents the dissipation from happening on codimension $1$ rectifiable sets, even if they are densely distributed. This argument generalizes the mechanism discovered in \cite{Shv09}, from which we have drawn inspiration, giving a more robust proof which applies to rougher singular sets. A key role is played by the incompressibility constraint which, together with the momentum equation, prevents the Bernoulli pressure from jumping where $u$ has a non-trivial normal component. Moreover, the $\mathcal H^d$-negligible set that might be present in $\Sigma$ is shown to be $D$-negligible by \cite{DDI24}*{Theorem 1.2}. In the companion paper \cite{InvVi24} these results will be generalized to the inhomogeneous setting, yet another instance of the robustness of the approach. In order to make the mechanism as clear as possible, we give a quite detailed heuristic in \cref{S:heuristic} below.

To summarize, the incompressibility constraint imposes severe limitations on the structure of dissipative sets. In particular, natural structures where the dissipation mechanism is active in the compressible setting, such as codimension $1$ rectifiable sets where solutions experience jumps, are excluded. That difficulty seems to have been noticed already by J. M. Burgers in his early works proposing the \quotes{Burgers' equation} as a simplified model for turbulent flows. In \cite{Burg48}*{Section XV}, after describing his failure in constructing a solution to the incompressible Navier--Stokes with non-trivial dissipation concentrated on a sheet, he comments:

\begin{changemargin}{1cm}{1cm} 
\begin{center}
\textit{In hydrodynamical turbulence [...] the geometrical features of incompressible fluid flow introduce specific complications into the problems, which, however, stand apart from the basic dynamical relations to which attention has been given here.}
\end{center}
\end{changemargin}

As a corollary of \cref{T:main_traces} we obtain the following result (see \cref{S:gmt} for precise definitions).

\begin{corollary}\label{C:main_SBD}
Let $u,p\in L^\infty_{x,t}$ be a weak solution to \eqref{E} with force $f \in L^1_{x,t}$. Assume that $u \in SBD_{x,t}$ in the sense of \cref{d: SBD} and $p$ has bilateral traces on Lipschitz space-time hypersurfaces in the sense of \cref{d: inner/outer trace}. Then \eqref{local_energy_eq} holds with $D\equiv 0$.
\end{corollary}

Note that $C^\infty$ fails to be dense in both $SBD$ and $L^\infty$. Thus, the mechanism that here prevents dissipation from happening is completely different from that in \cref{T:main_smooth_dense}. The main point is that for vector fields with bounded deformation, the dissipation is concentrated on the singular part of the symmetric gradient of the velocity field (see \cref{P: upper bound DR}). The  \quotes{Cantor part} is ruled out by the $SBD$ assumption, while the \quotes{jump part} is shown to be $D$-negligible by \cref{T:main_traces}. Differently from the usual commutator arguments (as the one of \cref{T:main_smooth_dense}), here both the incompressibility and the momentum equations are heavily used. The $SBD$ time regularity is needed for technical reasons and it does not seem possible to remove it without making some extra assumptions (see \cref{R:sbd_time} and \cref{P:SBD cons strange}).
We remark that a similar mechanism has been previously noted in \cite{ACM05} in the context of the transport equations. However, the fact that the density may not have strong traces and the absence of the pressure make the two proofs rather different.

\begin{remark}
    In \cref{C:main_SBD} the assumption on the existence of traces for the pressure could have been relaxed to hold only on the hypersurfaces where the singular part of the space-time symmetric gradient of the velocity field is concentrated. We are not aware of any non-trivial general result which guarantees the existence of strong traces for the pressure given the ones for the velocity.
\end{remark}

\begin{remark}
Since the condition $u\in BD$ can be equivalently formulated in terms of longitudinal increments (see \cref{l: characterization of BD}), by interpolating with $u\in L^\infty$, we obtain 
$$u\in BD\qquad \Longrightarrow \qquad  \int \abs{ \frac{z}{\abs{z}}\cdot\delta_z u(x)}^p dx\lesssim \abs{z} \qquad \forall p\geq 1, \quad \delta_z u (x):= u(x+z)-u(x). $$
Thus, \cref{C:main_SBD} offers the first energy conservation result in an Onsager critical class in which only an assumption on the longitudinal structure function is made. We remark that, by \cite{Driv22}*{Lemma 1}, the mere fact that $u$ solves \eqref{E} implies $u\in B^{\sfrac{1}{p}}_{p,\infty}$ for all $p\geq 2$, while for $p\in [1,2)$ the assumption stays only at the longitudinal level.
\end{remark}

\begin{remark}
We point out that a detailed analysis of the Duchon--Robert distribution associated to Euler solutions with measure first derivatives, or only some combination of them, has been recently developed by the second author in \cite{Inv25}. Among other things, it is shown that $D\equiv 0$ whenever $u\in L^\infty_{x,t} \cap L^1_t BD_x$ and $f,p \in L^1_{x,t}$. This certainly gives a cleaner, and conceptually stronger, version of \cref{C:main_SBD}. On the other hand, we stress the fact that the result from \cite{Inv25} relies on the full $BD$ spatial regularity of the velocity field, while \cref{T:main_traces} holds whenever the velocity and the pressure possess suitable traces. Hence,  our \cref{T:main_traces}  and \cite{Inv25}*{Theorem 1.2} are independent. 
\end{remark}

\subsection{Heuristic of \cref{T:main_traces}}\label{S:heuristic}
We describe the mechanism behind the proof \cref{T:main_traces}. To simplify and avoid technicalities, we focus on the stationary case 
\begin{equation}
    \label{E_stat}
    \left\{\begin{array}{l}
 \div (u \otimes u +  p I) =f \\
\div u=0.
\end{array}\right.
\end{equation}
By the theory of Measure-Divergence vector fields (see \cref{S:div_meas_fields}) the divergence of a vector field gives mass to a hypersurface $\Sigma$ if and only if its normal component exhibits a jump across $\Sigma$. Assume that $\Sigma$ is oriented by the normal vector $n_\Sigma:\Sigma\rightarrow \mathbb{S}^{d-1}$. The existence of strong traces (in the sense of \cref{d: inner/outer trace}) implies that the normal traces of $u$ on $\Sigma$ (from both sides) are given by $u^{\Sigma_\pm}\cdot n_{\Sigma}$ and that the normal trace of any non-linear function of $u$ will enjoy the composition formula (see \cref{c: composition trace}). Thus, $\div u=0$ implies 
\begin{equation}
    \label{intro_u.n_no_jump}
    u^{\Sigma_+}\cdot n_{\Sigma}=u^{\Sigma_-}\cdot n_{\Sigma}=:u_n \qquad \text{on } \Sigma.
\end{equation}
Also, since $f\in L^1$, the momentum equation implies 
\begin{equation}
    \label{intro_momentum_no_jump}
    u^{\Sigma_+} u_n +p^{\Sigma_+} n_{\Sigma}=u^{\Sigma_-}u_n + p^{\Sigma_-}n_{\Sigma} \qquad \text{on } \Sigma.
\end{equation}
If we scalar multiply \eqref{intro_momentum_no_jump} by $n_{\Sigma}$ we deduce that $u_n^2+p^{\Sigma_+}=u_n^2 + p^{\Sigma_-}$, from which 
$p^{\Sigma_+}=p^{\Sigma_-}$. Going back to \eqref{intro_momentum_no_jump}, the fact that the pressure does not jump implies 
\begin{equation}
    \label{intro_u_no_jump_on_u_no_0}
    u^{\Sigma_+}\equiv u^{\Sigma_-} \qquad \text{on }  \left\{ x\in \Sigma \, :\, u_n\neq 0\right\}.
\end{equation}
Recall that 
   \begin{equation} \label{local_energy_eq_stat}
\div  \left(\left(\frac{\abs{u}^2}{2}+p\right)u\right)=-D + f\cdot u \qquad \text{in } \mathcal D'.
\end{equation}
Since $f\cdot u\in L^1$, by \cref{P:div and normal trace}
$$ \abs{D}(\Sigma)=\int_{\Sigma}\abs{ \frac{\abs{u^{\Sigma_+}}^2}{2}u_n+ p^{\Sigma_+}- \frac{\abs{u^{\Sigma_-}}^2}{2}u_n-p^{\Sigma_-}}=\int_{\Sigma}\abs{ \frac{\abs{u^{\Sigma_+}}^2}{2}- \frac{\abs{u^{\Sigma_-}}^2}{2}} \abs{u_n} = 0, 
$$
because of \eqref{intro_u_no_jump_on_u_no_0}. To sum up, we have argued with the following stream of implications
\begin{align}
    &u\cdot n_{\Sigma} \text{ does not jump across } \Sigma \qquad \qquad\qquad  \qquad\qquad \text{ (by incompressibility)}\\
    \Longrightarrow \quad  &p \text{ does not jump across } \Sigma \qquad \qquad \qquad\qquad  \qquad \text{(by the momentum equation)}\\
    \Longrightarrow \quad  &u \text{ does not jump across } \left\{ x\in \Sigma \, :\, u_n\neq 0\right\} \qquad\,\, \text{ (by the momentum equation)}\\
      \Longrightarrow \quad  &D \text{ does not give mass to } \Sigma.
\end{align}
It is important to note that here a continuity assumption $u,p\in C^0$ would trivialize the result since it allows to prove $\abs{D}(\Sigma)=0$ directly from \eqref{local_energy_eq_stat} and \cref{P:div and normal trace}. In the time-dependent case, the non-trivial dynamic introduces an additional complication to the above reasoning because one has to distinguish the case in which the normal vector has (or not) a component along the time axes. However, besides some technicalities, as we shall see in \cref{S:cod_1_sing} one can follow essentially the same path.

\section{Tools}
We recall some notations and tools which are used throughout the paper.

\subsection{Some functional spaces and norms}\label{S:spaces}
In this section we will not specify the domain $\Omega$ in the definition of the norms, and we tacitly assume that everything is only defined locally inside $\Omega$ whenever we are not on the whole space or on the periodic box. The $BMO$ norm of a function $f$ is defined by  
$$ \norm{f}_{BMO}:=\sup_{x,\ell} \fint_{B_\ell (x)} \abs{ f(y)-\fint_{B_\ell (x)} f } \,dy. $$
The choice of balls instead of cubes gives an equivalent norm (see \cite{Stein93}*{page 140}). We say that $f \in VMO$ if 
$$
\lim_{\ell\rightarrow 0}\sup_{x} \fint_{B_\ell (x)} \abs{ f(y)-\fint_{B_\ell (x)} f}\,dy=0.
$$
It is well-known that $VMO :=\overline{C^\infty}^{\norm{\cdot}_{BMO}}$.

\begin{remark} \label{R:equivalent bmo}
By the John--Nirenberg inequality (see for instance \cite{Stein93}*{page 144}) it follows that 
    $$ \sup_{x,\ell} \fint_{B_\ell (x)} \abs{ f(y)-\fint_{B_\ell (x)} f}^p\,dy\leq  C_p\norm{f}_{BMO}^p\qquad \forall p<\infty. $$
\end{remark}
Following \cites{FW18,BGSTW19}, for any $\alpha\in (0,1)$ and any $p\in [1,\infty)$, we set 
$$
\norm{f}_{B^{\alpha}_{p,BMO}}:=\norm{ f}_{L^p} + \sup_{\ell>0} \frac{1}{\ell^\alpha}\left(\int \fint_{B_\ell (0)} \abs{f(x+y)-f(x)}^p \,dy \, dx \right)^\frac{1}{p}.
$$
Consequently, $B^{\alpha}_{p,VMO}$ is defined as the subspace of $B^{\alpha}_{p,BMO}$ in which
$$
\lim_{\ell\rightarrow 0} \frac{1}{\ell^\alpha}\left(\int \fint_{B_\ell (0)} \abs{f(x+y)-f(x)}^p \,dy \, dx \right)^\frac{1}{p}=0.
$$
A simple exercise shows that $B^{\alpha}_{p,VMO}$ is the closure of $C^\infty$ with respect to the $B^{\alpha}_{p,BMO}$ norm.

\begin{remark}
    In fact, the two norms in $B^{\alpha}_{p,\infty}$ and $B^{\alpha}_{p,BMO}$ are equivalent (see for instance \cite{DRP25}*{Remark 5.4}). In particular $B^\alpha_{p,VMO}=B^\alpha_{p,c_0}$.
\end{remark}
  
\subsection{Time-dependent curves of measures} 
Throughout this section, let $I \subset \R$ be an interval and let $\Omega \subset \R^d$ be an open set. We denote by $\mathcal{M}(\Omega; \R^m)$ the space of finite Borel measures with values in $\R^m$. Given $\mu \in \mathcal{M}(\Omega; \R^m)$, we denote by $\abs{\mu} \in \mathcal{M}(\Omega)$ its variation, i.e. the non-negative finite Borel measure defined by 
\begin{equation} \label{eq: variation of a measure}
\langle \abs{\mu}, \varphi \rangle := \sup_{\psi \in C^0(\Omega; \R^m), \abs{\psi} \leq \varphi} \int_\Omega \psi \cdot \, d  \mu \qquad \forall \varphi \in C^0_c (\Omega), \, \varphi \geq 0.
\end{equation}
We then  set $\norm{\mu}_{\mathcal{M}(\Omega)}: = \abs{\mu}(\Omega)$. If the dimension $m$ of the target space is clear from the context, we simply write $\mathcal{M}(\Omega)$ instead of $\mathcal{M}(\Omega; \R^m)$. We will often use the shorthand notation $\mathcal{M}_x$.

\begin{definition} \label{d: curve of measures}
We say that $ \mu = \{\mu_t\}_{t \in I}$ is a weakly measurable curve of measures if for almost every $t\in I$ we have $\mu_t \in \mathcal{M}(\Omega)$ and for any test function $\varphi \in C^0_c(\Omega)$ the map $t \mapsto \int_{\Omega} \varphi(x) \, d  \mu_t(x)$ is measurable. We say that $\mu \in L^1(I; \mathcal{M}(\Omega))$ if $\mu$ is weakly measurable and the map $t \mapsto \norm{\mu_t}_{\mathcal{M}_x}$ is in $L^1(I)$. We define the norm $\norm{\mu}_{L^1_t \mathcal{M}_x} : = \int_I \norm{\mu_t}_{\mathcal{M}_x}\,dt$.
\end{definition}

Curves of measures in $L^1(I; \mathcal{M}(\Omega))$ are in correspondence with finite Borel measures on $\Omega \times I$ which admit a disintegration with respect to the Lebesgue measure on $I$. 

\begin{lemma} \label{l: curve of measure 1}
Let $\mu = \{ \mu_t\}_{t \in I} \in L^1(I; \mathcal{M}(\Omega; \R^m))$. Then linear functional 
\begin{align} \label{eq: joint measure}
    \langle \mu_t \otimes \, d  t, \varphi\rangle : = \int_I \int_{\Omega} \varphi(x,t) \, d \mu_t(x) \, d  t  \qquad \forall \varphi \in C^0_c(\Omega \times (0,T))
 \end{align} 
is a well-defined finite Borel measure on $\Omega \times (0,T)$ and satisfies $(\pi_t)_{\#} \mu \ll \mathcal{L}^1 \llcorner I$, where $\pi_t: \Omega \times I \to I$ is the projection. Viceversa, let $\mu \in \mathcal{M}(\Omega \times I)$ such that $(\pi_t)_{\#} \mu \ll \mathcal{L}^1 \llcorner I$. Then there exists $\{ \mu_t\}_{t \in I} \in L^1(I; \mathcal{M}(\Omega))$ such that $\mu = \mu_t \otimes \, d  t$. Moreover the family $\{ \mu_t\}_{t \in I}$ is uniquely determined for a.e. $t \in I$. 
\end{lemma}

\begin{lemma} \label{l: curve of measure 2}
Let $\mu = \{ \mu_t\}_{t \in I} \in L^1(I;\mathcal{M}(\Omega))$. Then $\abs{\mu} = \{ \abs{\mu_t}\}_{t \in I}$ is in $L^1(I; \mathcal{M}(\Omega))$ and $\norm{\abs{\mu}}_{L^1_t \mathcal{M}_x} = \norm{\mu}_{L^1_t \mathcal{M}_x}$. For any $t \in I$, let $\mu_t = \mu_t^a + \mu_t^s$ be the Radon--Nikodym decomposition of $\mu_t$ with respect to $\mathcal{L}^d$, i.e. $\mu^a_t \ll \mathcal{L}^d$ and $\mu^s_t \perp \mathcal{L}^d$. Then $\{\mu^{a,s}_t\}_{t \in I} \in L^1(I; \mathcal{M}(\Omega))$ and, letting $\mu = \mu^a + \mu^s$ be the Radon--Nikodym decomposition of $\mu$ with respect to $\mathcal{L}^{d+1}$, it holds that $\mu^{a,s} = \mu^{a,s}_t \otimes \, d  t$ and $\norm{\mu^{a,s}}_{L^1_t\mathcal{M}_x} \leq \norm{\mu}_{L^1_t\mathcal{M}_x}$. 
\end{lemma}

The proof of the above lemmas is quite standard. See for instance \cite{InvVi24}*{Lemma 2.2 \& Lemma 2.3}.

\subsection{Measure-Divergence vector fields and traces}\label{S:div_meas_fields}

We recall the notion of the distributional normal trace \cites{ACM05,Shv09,CCT19,CTZ09}. Let $\Omega \subset \R^d$ be an open set. We say that a vector field $V \in L^\infty(\Omega;\R^d)$ is in $\mathcal{MD}^\infty(\Omega)$ if $\div V \in \mathcal{M}(\Omega)$. For such vector fields, it is possible to define the outer normal trace on $\partial \Omega$ coherently with the Gauss--Green formula
\begin{equation} \label{eq: normal distributional trace}
    \langle \tr_n(V, \partial \Omega), \varphi \rangle := \int_{\Omega} \nabla \varphi \cdot V \, d x + \int_{\Omega} \varphi \, d (\div V) \qquad \forall\varphi \in C^\infty_c(\R^d).
\end{equation}
The formula above defines a distribution supported on $\partial \Omega$. If $V \in C^1(\overline{ \Omega};\R^d)$ and $\Omega$ has Lipschitz boundary, $\tr_n(V, \partial \Omega)$ is induced by the integration on $\partial \Omega$ of $V \cdot n_{\partial \Omega}$, where $n_{\partial \Omega}$ is the outer normal. We always adopt the convention that the boundary of an open sets is oriented with the outer unit normal, denoted by $n_{\partial \Omega}$. It turns out that if $V \in \mathcal{MD}^\infty(\Omega)$ and $\Omega$ has Lipschitz boundary, then $\tr_n(V, \partial \Omega)$ is induced by an $L^\infty$ function on $\partial \Omega$, still denoted by $\tr_n(V, \partial \Omega)$ (see for instance \cite{ACM05}*{Proposition 3.2}). Moreover, this notion of trace is local in the sense that for any Borel set $\Sigma \subset \partial \Omega_1\cap \partial \Omega_2$ such that $n_{\partial \Omega_1}(x) = n_{\partial \Omega_2}(x)$ for $\mathcal{H}^{d-1}$-a.e. $x \in \Sigma$, where $\Omega_1, \Omega_2$ are Lipschitz open sets contained in $\Omega$, we have that the two traces coincide, i.e. $\tr_n(V, \partial \Omega_1) = \tr_n(V, \partial \Omega_2)$ for $\mathcal{H}^{d-1}$-a.e. $x \in \Sigma$ (see \cite{ACM05}*{Proposition 3.2}). This property allows one to define the distributional normal traces on an oriented Lipschitz hypersurface $\Sigma \subset \Omega$. Choose open sets $\Omega_1, \Omega_2$ such that $\Sigma \subset \partial \Omega_1 \cap \partial \Omega_2$ and $n_\Sigma(x) = n_{\partial \Omega_1}(x) = -n_{\partial \Omega_2}(x)$ for $\mathcal{H}^{d-1}$-a.e. $x \in \Sigma$ (see \cref{fig:trace}). Then, we define 
$$\tr_n(V, \Sigma_-) : = \tr_n(V, \partial \Omega_1), \qquad \tr_n(V, \Sigma_+) := -\tr_n(V, \partial \Omega_2) \qquad \text{ on } \Sigma.$$

\begin{figure}
\includegraphics[width=10cm,height=4.3cm]{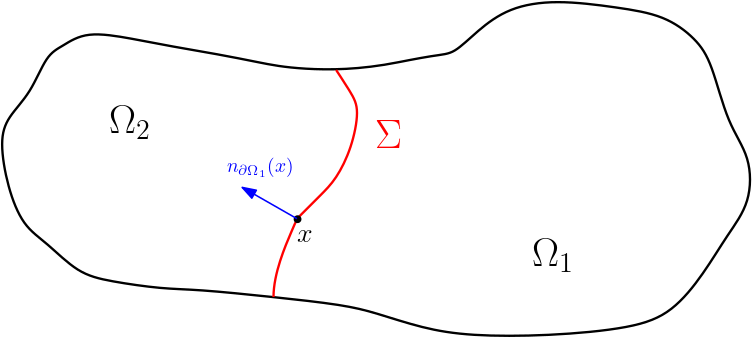}
\caption{Defining the distributional normal trace on $\Sigma$.}
\label{fig:trace}
\end{figure}

\begin{proposition}[\cite{ACM05}*{Proposition 3.4}]\label{P:div and normal trace} Let $V\in \mathcal{MD}^\infty(\Omega)$ and $\Sigma\subset \Omega$ be any oriented Lipschitz hypersurface. Then 
$$
\abs{\div V}(\Sigma)=\int_\Sigma \abs{\tr_n (V,\Sigma_+)- \tr_n (V,\Sigma_-)}  \, d\mathcal H^{d-1}.
$$
\end{proposition}

The notion of distributional normal trace is usually too weak to deal with non-linear problems. In particular, if $V$ has vanishing distributional normal trace and $\rho$ is a bounded scalar  function, it is not guaranteed that $\rho V$ has a distributional normal trace and, even if it does, it may not vanish. To overcome this problem, we recall a stronger notion of normal trace, which is inspired by the theory of $BV$ functions. We refer to  \cites{DRINV23,CDIN24,AFP00} for an extensive presentation.

\begin{definition} \label{d: inner/outer trace}
Let $\Omega \subset \R^d$ be an open set,  let $v \in L^\infty(\Omega ;\R^m)$ and let $\Sigma \subset \Omega$ be a Lipschitz hypersurface oriented by  $n_{\Sigma}$. We say that $v^{\Sigma_{\pm}} \in L^\infty(\Sigma ;\R^m)$ is the inner/outer trace on $\Sigma$ if for any sequence $r_j \to 0$ there exists an $\mathcal{H}^{d-1}$-negligible set $\mathcal{N}_{\pm} \subset \Sigma$ such that 
\begin{equation}
    \lim_{j \to \infty} \fint_{B_{r_j}^{\pm}(x)} \abs{v(z) - v^{\Sigma_{\pm}}(x)} \, d z = 0 \qquad \forall x \in \Sigma \setminus \mathcal{N}_{\pm},
\end{equation}
 where we set 
\begin{equation}
    B_{r}^{\pm}(x) := \left\{ z \in B_{r}(x) \,\colon \,\langle \pm\, n_{\Sigma}(x), z-x \rangle \geq 0 \right\}. 
\end{equation}
We say that $v$ has bilateral traces on $\Sigma$ if $v$ has both inner and outer traces on $\Sigma$. Similarly, given $V \in L^\infty(\Omega; \R^d)$, we say that $V^{\Sigma_{\pm}}_n \in L^\infty(\Sigma)$ is the inner/outer Lebesgue \emph{normal} trace on $\Sigma$ if for any sequence $r_j \to 0$ there exists an $\mathcal{H}^{d-1}$-negligible set $\mathcal{N}_{\pm} \subset \Sigma$ such that 
\begin{equation}
    \lim_{j \to \infty} \fint_{B_{r_j}^{\pm}(x)} \abs{V(z)\cdot n_\Sigma(x) - V^{\Sigma_{\pm}}_n(x)} \, d z = 0 \qquad \forall x \in \Sigma \setminus \mathcal{N}_{\pm}.
\end{equation}
\end{definition}

\begin{remark}
Concerning the definition of the normal Lebesgue trace, some remarks are in order.
\begin{itemize} 
\item The unit normal vector $n_\Sigma$ is defined $\mathcal{H}^{d-1}$-a.e on $\Sigma$, so we can assume it to be defined out of the negligible sets $\mathcal{N}_\pm$.
\item The term $n_\Sigma(x)$ can be replaced with the gradient of the \emph{signed} distance from $\Sigma$. More precisely, given $\Omega_1, \Omega_2$ as above, define the signed distance from $\Sigma$ by 
    \[
        \delta_{\Sigma}(y) := \begin{cases} 
                                -\dist(y,\Sigma), \text{ if } y\in\Omega_1 \\
                                \dist(y,\Sigma), \text{ if } y\in\Omega_2 .
                                \end{cases}
    \]
The signs are chosen in such a way that $\nabla\delta_\Sigma(y)$ is a local extension of $n_\Sigma(x)$ when $y$ is close to $x$. \cref{l: lebesgue point of distance} below implies that this replacement gives an equivalent definition.
\item  Following \cite{CDIN24}, one can also replace the average on half-balls $B^\pm_{r_j}(x)$ with an average on $\Omega_i\cap B_{r_j}(x)$. Indeed, at $\mathcal{H}^{d-1}$-a.e. $x \in \Sigma$ there exists a tangent plane $\Pi_x = \{n_\Sigma\}^\perp$ (in the usual sense) to $\Sigma$ and it holds\footnote{Here $\triangle$ denotes the symmetric difference between two sets, i.e. $A\triangle B= (A\setminus B) \cup (B\setminus A)$. }
\begin{align}
    \abs{(\Omega_1\cap B_{r_j}(x)) \triangle B_{r_j}^-(x)} = o(r_j^d) \qquad \text{and}\qquad 
    \abs{(\Omega_2\cap B_{r_j}(x)) \triangle B_{r_j}^+(x)} = o(r_j^d)
\end{align}
as $r_j\to 0$. Since only bounded functions are involved in the integrals, this also implies that the two different averaging procedures give the same result.
\end{itemize}
\end{remark}

By a straightforward modification of the proof of \cite{DRINV23}*{Theorem 2.4}, we deduce the following. 

\begin{lemma} \label{l: lebesgue point of distance}
Let $\Sigma \subset \R^d$ be a Lipschitz hypersurface oriented by a normal unit vector $n_\Sigma$. For $\mathcal{H}^{d-1}$-a.e. $x \in \Sigma$ it holds  
\begin{equation}
    \lim_{r \to 0} \fint_{B^{\pm}_r(x) } \abs{ \nabla \dist(y,\Sigma) \mp n_\Sigma(x) } \, d y = 0. 
\end{equation}

\end{lemma}
The trace of non-linear quantities of bounded vector fields enjoy the composition formula. 

\begin{lemma} \label{c: composition trace}
Let $\Omega \subset \R^d$ be an open set, let $\Sigma \subset \Omega$ be a Lipschitz hypersurface oriented by a normal unit vector $n_{\Sigma}$ and let $n,m \in \N$. Let $V \in L^\infty(\Omega; \R^n)$ be a vector field with inner/outer trace on $\Sigma$ according to \cref{d: inner/outer trace}. Then, for $g \in C(\R^n; \R^m)$ it holds that $g(V)$ has inner/outer trace on $\Sigma$ and 
\begin{equation} \label{eq: formula composition of trace}
 g(V)^{\Sigma_{\pm}} =  g(V^{\Sigma_{\pm}}).
\end{equation}
Moreover, if $n=m=d$, then the inner/outer normal Lebesgue trace of $g(V)$ on $\Sigma$ is given by 
\begin{equation} \label{eq: formula composition of normal trace}
 g(V)^{\Sigma_{\pm}}_n = g(V^{\Sigma_{\pm}}) \cdot n_\Sigma. 
\end{equation}
\end{lemma}

\begin{proof}
We focus on the inner case. Let $\omega: [0, +\infty) \to [0, +\infty)$ be the modulus of continuity of $g$ in the range of $V$. Without loss of generality we can assume that $\omega$ is concave. Then, for any sequence $r_j \to 0$, by Jensen's inequality
\begin{align}
    \fint_{B^+_{r_j}(x)} \abs{  g(V(z)) - g(V^{\Sigma_+}(x)) } \, d z & \leq \fint_{B^+_{r_j}(x)} \omega\left( \abs{V(z) - V^{\Sigma_+}(x)} \right) \, d z \\
    &\leq \omega \left( \fint_{B^+_{r_j}(x)}  \abs{V(z) - V^{\Sigma_+}(x)}  \, d z \right),
\end{align}
which vanishes as $j \to + \infty$ for $\mathcal{H}^{d-1}$-a.e. $x \in \Sigma$. Then, \eqref{eq: formula composition of trace} is proved. To conclude, \eqref{eq: formula composition of normal trace} directly follows by the definition.
\end{proof}

By \cite{CDIN24}*{Theorem 1.4},  the inner/outer normal Lebesgue trace agrees with the distributional normal one, whenever both exist.  

\begin{theorem} \label{T: distributional trace vs lebesgue trace}
Let $\Omega \subset \R^d$ be an open set, let $V \in L^\infty(\Omega; \R^d)$ such that $\div V \in \mathcal{M}(\Omega)$ and let $\Sigma \subset \Omega$ be a Lipschitz hypersurface oriented by $n_{\Sigma}$. Assume that $V$ has inner/outer normal Lebesgue trace on $\Sigma$ according to \cref{d: inner/outer trace}. Then, it holds 
\begin{equation}
    \tr_{n}(V, \Sigma_{\pm}) = V^{\Sigma_\pm}_n. 
\end{equation}
\end{theorem}

\subsection{Vector fields with bounded deformation} \label{S:gmt}

Here we recall some basic facts about the spaces of vector field with bounded deformation. We refer to \cites{TS80,ACDalM97, T83} for a detailed presentation of the topic. Given an open set $ \Omega \subset \R^d$, the space of functions with bounded deformation is given by vector fields whose symmetric gradient is a finite measure, i.e.
$$ BD(\Omega):=\left\{u\in L^1(\Omega;\R^d)\,: \, E u := \frac{\nabla u+\nabla u^T}{2}\in \mathcal{M}(\Omega;\R^{d\times d}) \right\}. $$
We say that $u \in BD_{\rm loc}(\Omega)$ if $u \in BD(O)$ for any open set $O \subset \joinrel \subset \Omega$. 
Vector fields with bounded deformation generalize functions with bounded variations, i.e. functions whose full gradient is a measure. Given $u: \Omega \to \R^d$, we denote by 
$$\delta_y u(x) := u(x+y)- u(x) \qquad x \in \Omega, \quad \abs{y} < \dist(x, \partial \Omega). $$ 
Vector fields with bounded deformation are characterized by the fact that sequences of difference quotients in \emph{longitudinal directions} stay  bounded in $L^1$. 

\begin{lemma} \label{l: characterization of BD} 
Let $\Omega \subset \R^d$ be an open set and $u \in L^1(\Omega; \R^d)$ such that 
\begin{equation} \label{eq: char BD}
    \sup_{\e < \dist(A, \partial \Omega)} \norm{ \e^{-1} y \cdot \delta_{\e y} u}_{L^1(A)} < +\infty \qquad \forall  A \subset \joinrel \subset \Omega \text{ open}, \quad \forall y \in B_1. 
\end{equation}
Then, $u \in BD_{\rm loc}(\Omega)$. Conversely, if $u \in BD_{\rm loc}(\Omega)$ then
\begin{equation}\label{BD_est_increment}
\norm{ \e^{-1} y \cdot \delta_{\eps y}u}_{L^1(A)}\leq \abs{y}^2 \abs{E u} \left( \overline{(A)_\eps}\right) \qquad \forall A \subset \joinrel \subset \Omega \text{ open}, \quad \forall \e < \dist(A, \partial \Omega),
\end{equation}
where $(A)_\eps := A + B_\e(0)$. 
\end{lemma}

\begin{proof}
Assume that \eqref{eq: char BD} holds. For any test function $\phi \in C^\infty_c(\Omega)$ we write 
\begin{align}
    \abs{\langle \phi, \partial_i u_j + \partial_j u_i \rangle} & = \abs{ \int_{\Omega} \big(\partial_i \phi(x) u_j(x) + \partial_j \phi(x) u_i(x) \big)  \, d x }\\
    &= \abs{ \lim_{\e \to 0} \int_{\Omega}  \frac{\delta_{-\e e_i} u_j(x) + \delta_{-\e e_j} u_i(x)}{\e} \phi(x) \, d x }. 
\end{align}
Trivial manipulations yield to 
\begin{equation}
    \delta_{-\e e_i} u_j(x) + \delta_{-\e e_j} u_i(x) = \delta_{-\e (e_i+e_j)} u(x) \cdot (e_i+ e_j) - \delta_{-\e e_i} u(x- \e e_j) \cdot e_i - \delta_{-\e e_j} u(x- \e e_i) \cdot e_j.
\end{equation}
Therefore, by \eqref{eq: char BD}  we deduce  
\begin{equation}
    \abs{\langle \phi, \partial_i u_j + \partial_j u_i \rangle} \leq C \norm{\phi}_{C^0(\Omega)}. 
\end{equation}
Thus, $(E u)_{ij} \in \mathcal{M}_{\rm loc}(\Omega)$ by Riesz theorem. Conversely, if $u \in BD_{\rm loc}(\Omega)$ then \eqref{BD_est_increment} holds (see e.g. \cite{ACM05}*{Lemma 2.4} for a proof). 
\end{proof}

We recall that a set $\Sigma \subset \R^d$ is said to be countably $\mathcal{H}^{d-1}$-rectifiable if $\mathcal{H}^{d-1}$-almost all of $\Sigma$ can be covered by countably many Lipschitz hypersurfaces $\Sigma_i \subset \R^d$, i.e.
$$\mathcal{H}^{d-1}\left( \Sigma \setminus  \bigcup_{i \in \N} \Sigma_i \right)=0. $$
To define an orientation on $\Sigma$, choose pairwise disjoint Borel sets $E_i$ and oriented hypersurfaces $\Sigma_i$ such that $E_i \subset \Sigma_i$ and $\bigcup_{i \in \N} E_i$ covers $\mathcal{H}^{d-1}$-almost all of $\Sigma$. Then, define 
$$n_\Sigma := n_{\Sigma_i} \text{ on } E_i  \qquad   \forall i.$$
This definition depends on the choice of the decomposition, but only up to a sign. Indeed, for any pair of oriented Lipschitz hypersurfaces $\Gamma, \Gamma'$ it holds  $n_{\Gamma'}(x) = \pm n_{\Gamma}(x)$ for $\mathcal{H}^{d-1}$-a.e. $x \in \Gamma \cap \Gamma'$. We will need the following property of the symmetric gradient of $BD$ vector fields (see \cite{ACDalM97}*{Remark 4.2, Theorem 4.3, Proposition 4.4}).

\begin{theorem} \label{thm structure of the gradient}
Let $\Omega \subset \R^d$ be an open set and let $u \in BD(\Omega)$. Let $E u = E^a u + E^s u$ be the Radon--Nikodym decomposition of $E u$ with respect to the Lebesgue measure. Then, there exists a canonical decomposition 
$$ E^s u = E^j u+ E^c u, $$
where $E^ju := E^s u\llcorner \mathcal J_u$ for a countably $\mathcal{H}^{d-1}$-rectifiable set $\mathcal{J}_u$ (oriented by a unit normal vector) and the measure $E^c u$ vanishes on $\mathcal H^{d-1}$-finite sets.
\end{theorem}

In the previous theorem the set $\mathcal J_u$ is the \emph{jump set} of $u$. Roughly speaking, the jump set is the set of points $x$ for which there exists a direction along which $u$ admits two distinct limits at $x$. The measure $E^ju$ is the \emph{jump part} of $Eu$ while $E^c u$ is called the \emph{Cantor part}. We define vector fields of \emph{special bounded deformation}.

\begin{definition} \label{d: SBD}
Let $\Omega \subset \R^d$ be  open  and  $u \in BD(\Omega)$. We say that $u\in SBD(\Omega)$  if $E^c u = 0$. In the case of a time dependent vector field $u:  \Omega\times (0,T) \to \R^d$, we say that $u\in SBD_{x,t}$ if $U :=(u,1) \in SBD (\Omega \times (0,T))$. 
\end{definition}

Vector fields of bounded deformation have traces on rectifiable sets (see \cite{T83}*{Chapter 2}, \cite{ACDalM97}*{Section 3}). 

\begin{theorem}[Boundary trace]\label{T:trace_in_BD}
Let $\Omega \subset \R^d$ be an open set and let $u \in BD(\Omega)$. Then, for any Lipschitz hypersurface $\Sigma \subset \Omega$ oriented by $n_\Sigma$, $u$ has bilateral traces (in the sense of \cref{d: inner/outer trace}) $u^{\Sigma_\pm}$ on  $\Sigma$ (with respect to $n_\Sigma$). Moreover
$$ E u \llcorner \Sigma =  \frac{1}{2} \Big( \left(u^{\Sigma_+} - u^{\Sigma_-}\right) \otimes n_\Sigma + n_\Sigma \otimes \left(u^{\Sigma_+}- u^{\Sigma_-}\right) \Big) \mathcal{H}^{d-1} \llcorner \Sigma. $$ 
\end{theorem}

\section{Critical classes and density of smooth functions}\label{S:smooth_dense}
Given two Banach spaces $X,Y$ we denote by $\mathcal{L}(X;Y)$ the space of bounded linear operators $L:X\rightarrow Y$. This is a Banach spaces when endowed with the norm
$$
\norm{L}_{\mathcal{L}}:=\sup_{\norm{v}_X\leq 1}\norm{Lv}_Y.
$$
The following fact is a direct consequence of the definition.
\begin{lemma}
    \label{L:linear_op}
    Let $\{L^\ell\}_{\ell>0}\subset \mathcal{L}(X;Y)$ be a sequence of linear and continuous operators between two Banach spaces $X$ and $Y$ such that $\sup_{\ell>0}\norm{L^\ell}_{\mathcal L }<\infty$. If there exists $Z\subset X$ such that 
    $$ \lim_{\ell \rightarrow 0} \norm{L^\ell v}_Y =0\qquad \forall v\in Z, $$
    then 
    $$ \lim_{\ell \rightarrow 0} \norm{L^\ell v}_Y =0\qquad \forall v\in \overline{Z}^{\norm{\cdot}_X}. $$
\end{lemma}

Clearly, the above property transfers to sequences of multilinear operators
$$ M^\ell: X_1\times\dots\times X_N\rightarrow Y. $$

\begin{lemma}
    \label{L:general conv to zero}
    Let $\{M^\ell\}_{\ell>0}$ be a family of multilinear operators for which there exists a constant $C>0$ such that
       \begin{equation}\label{T uniform bound}
\sup_{\ell>0}\norm{M^\ell[v_1,\dots,v_N]}_{Y}\leq C\prod_{i=1}^N \norm{v_i}_{X_i}\qquad \forall v_i\in X_i, \quad  i=1,\dots,N.
\end{equation}
Assume that there exists $i_0\in \{1,\dots,N\}$ and $Z \subset X_{i_0}$ such that 
\begin{equation}
    \label{T vanish on subset}
    \lim_{\ell\rightarrow 0} \norm{M^\ell[v_1,\dots,v_N]}_Y=0 \qquad \forall v_{i_0}\in Z,\quad \forall v_i\in X_i\, \text{ for } \, i\neq i_0.
\end{equation}
Then
\begin{equation}\label{T vanish on closure}
\lim_{\ell\rightarrow 0} \norm{M^\ell[v_1,\dots, v_N]}_Y=0 \qquad \forall v_{i_0}\in \overline{Z}^{\norm{\cdot}_{X_{i_0}}},\quad \forall v_i\in X_i\, \text{ for } \, i\neq i_0.
\end{equation}
\end{lemma}
\begin{proof}
Define the linear operators $L^\ell : X_{i_0}\rightarrow Y$ by $M^\ell$ applied to the \quotes{frozen} elements $v_i\in X_i$, for $i\neq i_0$, and where only the $i_0$-th entry is allowed to vary in $X_{i_0}$.  Because of \eqref{T uniform bound} and \eqref{T vanish on subset}, the sequence $\{L^\ell\}_{\ell>0}$ satisfies the assumptions of \cref{L:linear_op} and the conclusion follows. 
\end{proof}

Thus, if the goal is to prove energy conservation as a consequence of (the perhaps too strong property) $ \norm{D^\ell_\bullet}_{L^1_{x,t}}\rightarrow 0$, in view of \cref{L:general conv to zero} we only have to identify spaces which guarantee a uniform bound of $D^\ell_\bullet$ in $L^1_{x,t}$, i.e. \quotes{Onsager's critical spaces}, and then impose density of $C^\infty_x$ in any of the entries.
 Of course, the more general case of density dependent fluids (incompressible inhomogeneous and fully compressible) can be analogously considered.

\begin{proof}[Proof of \cref{T:main_smooth_dense}]
   Because of \eqref{uniform_bound_T} and \eqref{T_vanish_smooth}  the sequence of operators $\{T^\ell_{\bullet}\}_{\ell>0}$ satisfies the assumptions of \cref{L:general conv to zero} with $Z=C^\infty_x$. Thus, since we are assuming  $C^\infty_x$ to be dense in at least one of the spaces $X_i$, for $u\in L^{p_1}_t X_1\cap L^{p_2}_tX_2\cap L_t^{p_3}X_3$ we deduce  
    $$
    \lim_{\ell\rightarrow 0}\norm{ T^\ell_{\bullet} [u(t),u(t),u(t)]}_{L^1_x}=0 \qquad \text{for a.e. } t.
    $$
Also, \eqref{T uniform bound} implies
       $$
    \sup_{\ell >0 }\norm{ T^\ell_{\bullet} [u(t),u(t),u(t)]}_{L^1_x}\leq C \prod_{i=1}^3 \norm{u(t)}_{X_i}\in L^1_t, \qquad \text{since } \, \frac{1}{p_1}+ \frac{1}{p_2}+\frac{1}{p_3}=1.
    $$
By dominated convergence we conclude that $D^\ell_\bullet=T^\ell_\bullet [u(t),u(t),u(t)]\rightarrow 0$ in $L^1_{x,t}$.
\end{proof} 

As proved in \cites{BGSTW19,FW18}, the estimate
$$ \sup_{\ell>0}\norm{T^\ell_{\bullet}[v_1,v_2,v_3]}_{L^1}\leq C \prod_{i=1}^3 \norm{v_i}
    _{B^{\sfrac{1}{3}}_{3, BMO}}  \qquad \text{for } \, \bullet=\rm DR, \, CET $$
is a direct consequence of the definition of the $B^{\sfrac{1}{3}}_{3, BMO}$ norm. In particular, since $\overline{C^\infty}^{\norm{\cdot}_{B^{\sfrac{1}{3}}_{3, BMO}} } = B^{\sfrac{1}{3}}_{3, VMO}$, the corresponding energy conservation follows by \cref{T:main_smooth_dense}. In the next corollary we collect several classes which falls in this category, and thus  energy conservation follows from the uniform bound together with $C^\infty_x$ being dense in at least one of the entries.

\begin{corollary}\label{C: en cons euler general}
   $D^\ell_{\rm{DR}}\rightarrow 0$ in $L^1_{x,t}$ under any of the following assumptions:
    \begin{itemize}
        \item[(i)] $u\in L^3_t B^{\sfrac{1}{3}}_{3,VMO}$;
        \item[(ii)] $u \in L^{p_1}_t BD_x\cap L^{p_2}_t C^0_x$ with $\frac{1}{p_1}+\frac{2}{p_2}=1$;
        \item[(iii)] $u \in L^{p_2}_t L^\infty_x$ with $E u:=\frac{\nabla u + \nabla u^T}{2} \in L^{p_1}_t L^1_x $,  where $\frac{1}{p_1}+\frac{2}{p_2}=1$;
        \item[(iv)] $u \in  L^{p_2}_t L^2_x$ with  $ Eu\in L^{p_1}_t L^\infty_x$, where $\frac{1}{p_1}+\frac{2}{p_2}=1$.
    \end{itemize}
    Similarly, $D^\ell_{\rm CET}\rightarrow 0$ in $L^1_{x,t}$
    under any of the following assumptions:
    \begin{itemize}
        \item[(v)] $u\in L^3_t B^{\sfrac{1}{3}}_{3,VMO}$;
        \item[(vi)] $u \in L^{p_1}_t BD_x\cap L^{p_2}_t VMO_x$ with $\frac{1}{p_1}+\frac{2}{p_2}=1$;
        \item[(vii)] $u \in  L^{p_2}_t L^2_x$ with $E u \in L^{p_1}_t L^\infty_x$, with  $\frac{1}{p_1}+\frac{2}{p_2}=1$.
    \end{itemize}
\end{corollary}

The uniform bounds of $\{T^\ell_\bullet\}_{\ell>0}$ in all the  norms above are immediate, probably with the only exception of $(vi)$, for which we give a proof in \cref{P:bmo_est} below.
 
\begin{proposition}\label{P:bmo_est}
   There exists a constant $C>0$ such that 
    \begin{equation}\label{est_bmo_cet}
    \sup_{\ell>0} \norm{T^\ell_{\rm CET}[v,v,u]}_{L^1}\leq C \norm{v}^2_{BMO} \norm{u}_{BD}.
    \end{equation}
\end{proposition}
As it is clear from the proof, the above bound does not seem to be available for $\{T^\ell_{\rm DR}\}_{\ell>0}$.
\begin{proof}
    Recall that, since $v_\ell\otimes v_\ell-(v\otimes v)_\ell$ is symmetric, we can write
    $$
T^\ell_{\rm CET}[v,v,u]=(v_\ell\otimes v_\ell-(v\otimes v)_\ell):Eu_\ell.
$$
Since $\norm{Eu_\ell}_{L^1} \leq \norm{u}_{BD}$, we get
$$
\norm{T^\ell_{\rm CET}[v,v,u]}_{L^1}\leq \norm{ v_\ell\otimes v_\ell-(v\otimes v)_\ell}_{L^\infty} \norm{u}_{BD}.
$$
Rewriting 
$$
((v\otimes v)_\ell -v_\ell\otimes v_\ell)(x)=\int_{B_\ell (x)} (v(y)-v_\ell(x))\otimes (v(y)-v_\ell(x))\rho_\ell (x-y)\,dy,
$$
by adding and subtracting $\fint_{B_\ell (x)} v$ we obtain
\begin{align}
\abs{((v\otimes v)_\ell -v_\ell\otimes v_\ell)(x)} &\leq C \fint_{B_\ell(x)} \abs{v(y)-v_\ell (x)}^2 dy\\
&\leq C \left( \fint_{B_\ell(x)} \abs{v(y)-\fint_{B_\ell (x)} v}^2 dy + \abs{\fint_{B_\ell (x)} v(y)\,dy - v_\ell (x)}^2\right)\\
&\leq C  \fint_{B_\ell(x)} \abs{v(y)-\fint_{B_\ell (x)} v}^2 dy \leq C \norm{v}_{BMO}^2,
\end{align}
where the last bound follows by the John--Nirenberg inequality (see \cref{R:equivalent bmo}).
\end{proof}

\begin{remark}
    The meticulous reader may notice that the estimate \eqref{est_bmo_cet} for $\{T^\ell_{\rm CET}\}_{\ell>0}$ is proved only when the first two entries coincide. This is indeed necessary if we want to see only the symmetric part of the gradient of $u$. However, by (tri)linearity this does not affect the argument 
\begin{changemargin}{1cm}{1cm} 
\begin{center}
\textit{uniform bound $+$ vanishing on $C^\infty$ $\Longrightarrow$  vanishing on the closure of $C^\infty$ in the critical norm.}
\end{center}
\end{changemargin}
The same reasoning applies to $(vii)$ in \cref{C: en cons euler general}.
\end{remark}

\begin{remark}\label{S:transport_renorm}
Since the form of the commutator falls in the same category, the argument applies to the \quotes{renormalization property} for the continuity equation as well. The failure of density of $C^\infty_x$ is the substantial difference between the proof by Ambrosio \cite{Ambr04} for $BV_x$ vector fields and that by Diperna--Lions \cite{DipLi89} in $W^{1,1}_x$. In both cases, the density is assumed to be bounded. The idea of Ambrosio is to note that the geometric structure of $BV_x$ vector fields allows to prove that the commutator is weakly vanishing, although it might fail to converge strongly to $0$. For instance, the argument presented in this note recovers the case of $BD_x$ vector fields with density in $C^0_x$. The case of $BV_x$ vector fields with continuous density was addressed in \cite{CL02}.
\end{remark}

\section{Dissipation for co-dimension 1 singular structures}\label{S:cod_1_sing}

In this section we prove \cref{T:main_traces} and \cref{C:main_SBD}.

\begin{proof} [Proof of \cref{T:main_traces}]
To show that $D$ vanishes on countably $\mathcal{H}^d$-rectifiable sets, we have to check that 
\begin{align}
     |D|(\mathcal{N}) &= 0 \qquad \text{if } \mathcal{H}^d(\mathcal{N}) = 0,\label{eq: D<< H^d}\\
      \abs{D}(\Sigma) &= 0  \qquad \text{for any }\Sigma \text{ oriented Lipschitz hypersurface}.\label{eq: D = 0 on hypersurfaces}
\end{align}
Since $D$ is a Radon measure, we have 
$$ V : = \left( u\left( \frac{\abs{u}^2}{2} + p \right), \frac{\abs{u}^2}{2} \right) \in \mathcal{MD}^\infty_{\rm loc}(\Omega \times (0,T)). $$
Since $u,p \in L^\infty_{x,t}$ and $f \in L^1_{x,t}$, by \cite{Sil05}*{Theorem 3.2}    we infer  $D \ll \mathcal{H}^d$, thus proving \eqref{eq: D<< H^d} (see also \cite{DDI24}*{Theorem 3.1} for the case $D\geq0$). From now on, we fix a Lipschitz hypersurface $\Sigma$ oriented by a unit vector $n_\Sigma= (n_x, n_t) \in \R^d \times \R$. Since $f \cdot u\in L^1_{x,t}$, by \cref{P:div and normal trace} we deduce  
\begin{equation}
\abs{D}(\Sigma) =  \abs{- \div_{x,t} V + f \cdot u } (\Sigma)  = \abs{\div_{x,t} V}  (\Sigma) = \int_{\Sigma}   \abs{\tr_n(V, \Sigma_+) - \tr_n(V, \Sigma_-)} \, d \mathcal{H}^{d}, 
\end{equation} 
where $\tr_n(V, \Sigma_{\pm})$ is the inner/outer distributional normal trace of $V$ on $\Sigma$. Then, by \cref{T: distributional trace vs lebesgue trace}, we infer that $\abs{D}(\Sigma)=0$ if and only if
\begin{equation} 
    V^{\Sigma_+}_n(y) = V^{\Sigma_-}_n(y) \qquad \text{for $\mathcal{H}^d$-a.e. } y=(x,t)\in \Sigma. 
\end{equation} 
For simplicity, we denote by $u^{\Sigma_{\pm}} = : u^\pm$ and $ p^{\Sigma_{\pm}} = : p^{\pm}$. By \cref{c: composition trace} we have
\begin{equation}
    V_n^{\Sigma_\pm} =  \left( \frac{\abs{u^\pm }^2}{2} + p^{\pm} \right) u^{\pm } \cdot n_x + \frac{\abs{u^\pm}^2}{2} n_t. \label{eq: trace V} 
\end{equation} 
Therefore, we have to check that 
\begin{equation} \label{eq: trace does not jump}
     \left( \frac{\abs{u^+ }^2}{2} + p^{+} \right) u^{+} \cdot n_x + \frac{\abs{u^+}^2}{2} n_t =  \left( \frac{\abs{u^- }^2}{2} + p^{-} \right)u^{- } \cdot n_x + \frac{\abs{u^-}^2}{2} n_t \qquad \text{for $\mathcal{H}^d$-a.e. } y \in \Sigma. 
\end{equation}
Since $u,p$ solves \eqref{E} and $f\in L^1_{x,t}$, setting $W : = ( u \otimes u + \textrm{Id } p, u ) $, by \cref{P:div and normal trace} it holds 
\begin{equation}
    \int_{\Sigma} \abs{ \tr_n(W, \Sigma_+) - \tr_n(W, \Sigma_-) } \, d \mathcal{H}^{d} = \abs{ \div_{x,t} W } ( \Sigma ) =  \abs{f} (\Sigma) = 0,
\end{equation}
where the last equality follows by $f\in L^1_{x,t}$. Hence, we infer that $\tr_n(W, \Sigma_+) \equiv \tr_n(W, \Sigma_-) $ and, by \cref{c: composition trace} and \cref{T: distributional trace vs lebesgue trace}, we deduce 
\begin{equation} \label{eq: euler trace 1}
    u^+ (u^+ \cdot n_x) + p^+ n_x + u^+ n_t =  W^{\Sigma_+}_n = W^{\Sigma_-}_n =  u^- (u^- \cdot n_x) + p^- n_x + u^- n_t \qquad \text{for $\mathcal{H}^d$-a.e. } y \in \Sigma.
\end{equation}
Similarly, by \cref{c: composition trace} and \cref{T: distributional trace vs lebesgue trace}, since $U : = (u,1)$ is divergence free, we have
\begin{equation}
    u^+ \cdot n_x + n_t = U^{\Sigma_+}_n =  U^{\Sigma_-}_n = u^- \cdot n_x + n_t \qquad \text{for $\mathcal{H}^d$-a.e. } y \in \Sigma, 
\end{equation}
that is 
\begin{equation} \label{eq: euler trace 2}
    u^+ \cdot n_x = u^- \cdot n_x=:u\cdot n_x \qquad \text{for $\mathcal{H}^d$-a.e. } y \in \Sigma.
\end{equation}
Taking the scalar product of \eqref{eq: euler trace 1} with $n_x$ and using \eqref{eq: euler trace 2}, we find 
\begin{equation} \label{eq: euler trace 3}
    p^+ \abs{n_x}^2 = p^- \abs{n_x}^2 \qquad \text{for $\mathcal{H}^d$-a.e. } y \in \Sigma. 
\end{equation}

Let us decompose $\Sigma=\Sigma_1\cup \Sigma_2\cup \Sigma_3$, with 
\begin{align}
    \Sigma_1 &:= \left\{ y=(x,t)\in \Sigma \, :\, n_x(y) \neq 0 \,\text{ and } \, (u\cdot n_x)(y)\neq 0 \right\},\\
     \Sigma_2 &:= \left\{ y=(x,t)\in \Sigma \, :\, n_x(y) \neq 0 \,\text{ and } \, (u\cdot n_x)(y)= 0 \right\},\\
     \Sigma_3 &:= \left\{ y=(x,t)\in \Sigma \, :\, n_x(y) = 0 \right\}.
\end{align}
To conclude, we have to check that \eqref{eq: trace does not jump} holds on every $\Sigma_i$, for $i=1,2,3$.

\underline{\textsc{On $\Sigma_1$}}: Since $n_x\neq 0$, by \eqref{eq: euler trace 3} we deduce $p^+\equiv p^-$ on $\Sigma_1$, and \eqref{eq: euler trace 1} becomes 
$$
\left( u^+ - u^-\right) \big( u\cdot n_x + n_t\big) =0 \qquad \text{on } \Sigma_1.
$$
Thus, either $y\in \Sigma_1$ is such that $u^+(y)=u^-(y)$, or $(u\cdot n_x)(y)=-n_t(y)$. In both cases (recall that here $p^+\equiv p^-$) \eqref{eq: trace does not jump} holds.

\underline{\textsc{On $\Sigma_2$}}: Since $n_x(y)\neq 0$, also in this case \eqref{eq: euler trace 3} implies $p^+\equiv p^-$ on $\Sigma_2$. Then, the fact that $(u\cdot n_x)(y)=0$ reduces \eqref{eq: euler trace 1} to $u^+ n_t=u^- n_t$. Then, either $y\in \Sigma_2$ is such that $u^+(y)=u^-(y)$, or $n_t(y)=0$. It is a direct check that \eqref{eq: trace does not jump} holds in both situations.

\underline{\textsc{On $\Sigma_3$}}: Since $(n_x,n_t)$ is a unit vector, it must be $n_t=\pm 1$. Thus \eqref{eq: euler trace 1} becomes $u^+ =u^-$ on $\Sigma_3$ and \eqref{eq: trace does not jump} holds again.
\end{proof}

The following proposition provides a useful result to deal with bounded solutions whose longitudinal spatial increment scale linearly when measured in $L^1$.

\begin{proposition} \label{P: upper bound DR}
Let $u \in L^\infty_{x,t} $ and $p \in L^1_{x,t}$ be a weak solution to \eqref{E} with $f \in L^1_{x,t}$. Assume that 
$$ (E u)_{ij} :=  \frac{\partial_i u_j + \partial_j u_i}{2} \in \mathcal{M}_{x,t} \qquad \text{ for any } i,j=1,\dots,d. $$ 
Then, the Duchon--Robert distribution $D$ defined by \eqref{local_energy_eq} is a Radon measure and, letting $E^s u$ be the singular part of $E u$ with respect to the space-time Lebesgue measure, locally inside $\Omega\times(0,T)$ there exists a constant $C>0$ such that
\begin{equation}
    \abs{D}\leq C \abs{E^s u} \qquad \text{in } \mathcal M_{x,t}. \label{eq: upper bound DR} 
\end{equation}
\end{proposition}

\begin{proof}
Recall that $D=\lim_{\ell\rightarrow 0} R_\ell :\nabla u_\ell$ in the sense of distributions and, since $R_\ell$ is a symmetric matrix, it holds that $R_\ell :\nabla u_\ell=R_\ell :E u_\ell$. Although we are mollifying only in space, it is  a direct  check that $E u_\ell = (E u)_\ell$ in $\mathcal{D}'_{x,t}$ and, since $u\in L^1_{x,t}$, then $(E u)_\ell\in L^1_{x,t}$. Therefore, we infer that $\sup_{\ell>0} \norm{E u_\ell}_{L^1_{x,t}}<\infty$. Together with the assumption $u\in L^\infty_{x,t}$, this yields to $\sup_{\ell>0} \norm{R_\ell:E u_\ell}_{L^1_{x,t}}<\infty$. In particular $D$ is a Radon measure. By the Radon--Nikodym decomposition we write $E u =E^a u + E^s u$ with $E^a u \in L^1_{x,t}$ and $E^s u$ singular with respect to $\mathcal{L}^{d+1}$. Then, given $\varphi \in C^\infty_{x,t}$ with compact support, we split\footnote{Since $(E u)_\ell, (E^a u)_\ell \in L^1_{x,t}$, also $(E^s u)_\ell \in L^1_{x,t}$.}
\begin{align}
    \abs{ \int \varphi\, dD}&=\lim_{\ell \rightarrow 0} \abs{ \int \varphi R_\ell : (E u)_\ell\, dx\,dt} \\
    &\leq \limsup_{\ell\rightarrow 0} \int \abs{\varphi} \abs{R_\ell: (E u)_\ell} \,dx\,dt \\
    & \leq \limsup_{\ell\rightarrow 0}\left(\int \abs{\varphi} \abs{R_\ell:(E^a u)_\ell} \,dx\,dt + \int \abs{ \varphi} \abs{R_\ell:(E^s u)_\ell} \,dx\,dt\right).\label{split_ac_sing}
\end{align}
We bound the first term by 
$$ \int \abs{\varphi} \abs{R_\ell: (E^a u)_\ell} \,dx\,dt \leq \int \abs{\varphi} \abs{R_\ell: ( (E^a u)_\ell-E^a u) }\,dx\,dt + \int \abs{\varphi} \abs{R_\ell:E^a u} \,dx\,dt=:I_\ell + II_\ell. 
$$
Since $E^a u\in L^1_{x,t}$, by the standard property of the convolution, we have that $(E^a u)_\ell (\cdot, t) \to (E^a u) (\cdot, t)$ in $L^1_x$ for a.e. $t$ and, locally in space and time,
$$\sup_{\ell >0} \norm{(E^a u)_\ell(\cdot, t)}_{L^1_x} \leq \norm{(E^a u)(\cdot, t)}_{L^1_x} \in L^1_t. $$
Therefore, recalling that $u \in L^\infty_{x,t}$, by the dominated convergence theorem we get that $I_\ell \to 0$ as $\ell \to 0$. To estimate $II_\ell$, we recall that $\sup_{\ell>0} \abs{R_\ell:E^a u}\leq C\abs{E^a u}\in L^1_{x,t}$ and $\abs{R_\ell}\rightarrow 0$ almost everywhere. Thus, again by the dominated convergence theorem, we deduce $II_\ell\rightarrow 0$ as $\ell\rightarrow 0$. Therefore, \eqref{split_ac_sing} reduces to 
$$ \abs{\int \varphi\, dD} \leq \limsup_{\ell\rightarrow 0}\int \abs{\varphi} \abs{R_\ell:(E^s u)_\ell}\,dx\,dt\leq C \norm{\phi}_{C^0_{x,t}} \limsup_{\ell\rightarrow 0}\int_{\spt \phi }  \abs{(E^s u)_\ell} \,dx\,dt. 
$$
Although we are mollifying only in space, it holds $(E^s u)_\ell \rightharpoonup^* E^s u$ in $\mathcal{D}'_{x,t}$. Therefore 
\begin{equation}
    \abs{\int \phi \, d D} \leq C \norm{\phi}_{C^0_{x,t}} \abs{E^s u}(\spt \phi ) \qquad \forall \phi \in C^0_c(\Omega \times (0,T)), 
\end{equation}
thus proving \eqref{eq: upper bound DR}. 
\end{proof}

We are in the position to prove \cref{C:main_SBD}.

\begin{proof} [Proof of \cref{C:main_SBD}]
To avoid confusion, we will use the notation $E_{x,t}$ when the symmetric gradient is computed in the space-time variables, while $E$ for the spatial variable only. Since $U:=(u,1)\in SBD_{x,t}$ (see \cref{d: SBD}), we have\footnote{Here we adopt the usual convention that $u$ is a column vector.}
$$
E_{x,t} U =\begin{pmatrix}
        E u  & \frac{\partial_t u}{2}
        \\ \frac{\partial_t u^T}{2}  & 0
    \end{pmatrix} \in \mathcal{M}_{x,t}.
$$
Therefore, by \cref{P: upper bound DR} we deduce that $D$ is a Radon measure and, since $\abs{E^s u}\leq \abs{ E^s_{x,t} U}$ in the sense of measures\footnote{$E^s_{x,t} U$ is the singular part of $E_{x,t} U$ with respect to $\mathcal L^{d+1}$.}, by \eqref{eq: upper bound DR} we infer that 
$$\abs{D} \leq C \abs{E^s_{x,t} U} $$
holds locally inside $\Omega\times (0,T)$.
Since $U\in SBD_{x,t}$, by \cref{thm structure of the gradient} we find a countably $\mathcal H^d$-rectifiable set $\Sigma\subset \Omega\times (0,T)$ such that $\abs{E^s_{x,t} U} =\abs{E^s_{x,t} U}\llcorner \Sigma$. Then $\abs{D}=\abs{D}\llcorner \Sigma$ necessarily. By \cref{T:trace_in_BD} the vector field $U$, and thus $u$ too, has bilateral traces on oriented Lipschitz hypersurfaces. Then, \cref{T:main_traces} implies $\abs{D}(\Sigma)=0$ and, since $\abs{D}$ is concentrated on $\Sigma$, the proof is concluded. 
\end{proof}

Although some steps in the previous proof would still hold under the assumption $u\in L^1_tSBD_x$ in place of $u\in SBD_{x,t}$, without the time regularity assumption it does not seem to be possible anymore to deduce $D\equiv 0$ (see \cref{R:sbd_time} below). Note that strictly speaking, none of the two assumptions is stronger than the other. Indeed,  $u\in L^1_tSBD_x$ implies that the time marginal of the measure $E u$ is absolutely continuous with respect to Lebesgue, a property that might fail to be true if $u\in SBD_{x,t}$. For these reasons, we also state the following.

\begin{corollary} \label{P:SBD cons strange}
Let $u \in L^\infty_{x,t}\cap L^1_t BD_{x}$ and $p \in L^1_{x,t}$ be a weak solution to \eqref{E} with  $f \in L^1_{x,t}$. Then, the Duchon--Robert distribution $D$ defined by \eqref{local_energy_eq} is a Radon measure and there exists a constant $C>0$ such that
$$ \abs{D}\leq C \abs{E^s u_t}\otimes dt \qquad \text{in } \mathcal M_{x,t}. $$
\end{corollary}

In particular, if in the above corollary we would additionally require that  $u$ and $p$ have bilateral traces on Lipschitz space-time hypersurfaces in the sense of \cref{d: inner/outer trace} and $\abs{E^s_x u_t}\otimes dt $ is concentrated on a countably $\mathcal{H}^d$-rectifiable set, we would again deduce $D\equiv 0$. However, $u\in L^1_tSBD_x$ does not seem to be sufficient to guarantee the latter property (see \cref{R:sbd_time}). The proof of \cref{P:SBD cons strange} is similar  to that of \cref{C:main_SBD}. Since the time variable is treated in a different way, we prove it for completeness.

\begin{proof}
Let $E u_t$ be the symmetric gradient of $u(\cdot, t)$. Since $u \in L^1_{x,t}$, we have that $\{ E u_t\}_{t \in (0,T)}$ is a curve of measures in the sense of \cref{d: curve of measures} and, by \cref{l: curve of measure 1} it can be identified with $E u = E u_t \otimes dt \in \mathcal{M}_{x,t}$. Here the notation is consistent with that of \cref{P: upper bound DR}, since it can be directly checked that $E u_t \otimes dt$ is the distributional symmetric gradient of $u$ with respect to the spatial variables. Then, the conclusion follows by \cref{P: upper bound DR} and \cref{l: curve of measure 2}, since 
$$\abs{D} \leq C \abs{E^s u} = C \abs{E^s u_t} \otimes dt. $$
\end{proof}

\begin{remark} \label{R:sbd_time}
Let us comment on the $SBD$ space-time assumption in \cref{C:main_SBD}. Assume we replace that assumption with $u \in L^1_t SBD_x$. In this case, \cref{P:SBD cons strange} implies that $\abs{E^s u_t}$, and thus $D$ too, is concentrated on a countably $\mathcal H^{d-1}$-rectifiable set, for a.e. time $t$. However, the lack of any control on the dynamics of these sets does not allow us to deduce that the resulting space-time concentration set of $D$ is countably $\mathcal H^d$-rectifiable. In particular, the traces argument from \cref{T:main_traces} cannot be applied anymore, even with the additional requirement that $u$ has bilateral traces on Lipschitz oriented hypersurfaces in space-time. As we have seen in the proof of \cref{C:main_SBD}, the $SBD_{x,t}$ assumption is a way to prevent the wild time evolution of the support of the dissipation.
\end{remark}

\begin{remark}
    It is not clear to the authors if $u\in SBD_{x,t}$ could be deduced from $u\in L^1_t SBD_x\cap L^\infty_{x,t}$ by using the PDE. Although it is possible to rewrite it as 
    $$
    \partial_t u_j = - u_i \left( \partial_i u_j +\partial_j u_i \right) - \partial_j \left( p - \frac{|u|^2}{2}\right)\qquad \forall j=1,\dots,d,
    $$
    it seems unlikely that $\partial_j\left(p - \frac{|u|^2}{2}\right)$ can be proved to be a locally finite measure in space-time.
\end{remark}

\bibliographystyle{plain} 
\bibliography{biblio}

\end{document}